\documentclass[ECP]{ejpecp} % replace ECP by EJP if needed.

\AUTHORS{%
  Dmitry~Ostrovsky%%\footnote{195 Idlewood Dr., Stamford, CT 06905, USA.
}

\SHORTTITLE{Barnes Beta Distributions}

\TITLE{Theory of Barnes Beta Distributions
}
 % Separate items with ;

\AMSSUBJ{11M32; 30D05; 33B15; 41A60; 60E07; 60E10} % Separate items with ;
\VOLUME{0} \YEAR{2013} \PAPERNUM{0} \DOI{vVOL-PID}

\ABSTRACT{A new family of probability distributions $\beta_{M, N},$
$M=0\cdots N,$ $N\in\mathbb{N}$ on the unit interval $(0, 1]$ is
defined by the Mellin transform. The Mellin transform of $\beta_{M,
N}$ is characterized in terms of products of ratios of Barnes
multiple gamma functions, shown to satisfy a functional equation,
and a Shintani-type infinite product factorization. The distribution
$\log\beta_{M, N}$ is infinitely divisible. If $M<N,$
$-\log\beta_{M, N}$ is compound Poisson, if $M=N,$ $\log\beta_{M,
N}$ is absolutely continuous. The integral moments of $\beta_{M, N}$
are expressed as Selberg-type products of multiple gamma functions.
The asymptotic behavior of the Mellin transform is derived and used
to prove an inequality involving multiple gamma functions and
establish positivity of a class of alternating power series. For
application, the Selberg integral is interpreted probabilistically
as a transformation of $\beta_{1, 1}$ into a product of
$\beta^{-1}_{2, 2}s.$}

\begin{document}

This paper addresses the question of what type of probability
distributions\footnote{The terms ``probability distribution'' and
``random variable'' are used interchangeably in this paper.} one can
obtain by considering products of ratios of Barnes multiple gamma
functions \cite{mBarnes}. Classically, it is well-known that the
Mellin transform\footnote{It is more natural to define the Mellin
transform as $\int_0^\infty x^q\,f(x)\,dx$ as opposed to the usual
$\int_0^\infty x^{q-1}\,f(x)\,dx$ for our purposes.} of the beta
distribution $\beta$ with parameters $b_0, b_1>0$ and density
$\Gamma(b_0+b_1)/\Gamma(b_0)\Gamma(b_1) \, x^{b_0-1}(1-x)^{b_1-1}$
is given by
\begin{equation}\nonumber
{\bf E}[\beta^q] =
\frac{\Gamma(q+b_0)}{\Gamma(b_0)}\frac{\Gamma(b_0+b_1)}{\Gamma(q+b_0+b_1)},
\,\,\Re(q)>-b_0.
\end{equation}
%%\frac{\Gamma(b_0+b_1)}{\Gamma(b_0)\Gamma(b_1)}\int\limits_0^1
%%x^{q+b_0-1}(1-x)^{b_1-1}\,dx =

The main contribution of this paper is to construct and study the
main properties of a novel family of probability distributions on
the unit interval $(0, 1]$ that naturally generalize the beta
distribution to arbitrary multiple gamma functions. In particular,
we show that all of these distributions have infinitely divisible
logarithm, satisfy a functional equation and several symmetries, and
admit a remarkable infinite-product factorization. We call them
Barnes beta distributions.

Our paper contributes to several areas of current interest in
probability theory. First, we contribute to the probabilistic study
of Barnes multiple gamma functions and, more generally, the study of
infinite divisibility in the context of special functions of
analytic number theory complementing \cite{BiaPitYor}, \cite{Jacod},
\cite{LagRains}, \cite{NikYor}. We show that a new class of
L\'evy-Khinchine representations is naturally associated with
multiple gamma functions. Moreover, the meromorphic functions that
we introduce as the Mellin transform of the Barnes beta
distributions appear to have a number analytic significance as their
pole structure depends on the rationality of the parameters of the
distribution.

Second, there is a long-standing interest in the literature in the
study of Dufresne distributions, whose defining property is that
their Mellin transform is given in the form of a product of ratios
of Euler's gamma functions, confer \cite{ChaLet}, \cite{Duf10}, and
references therein. In addition, there have recently appeared a
series of papers \cite{HubKuz}, \cite{Kuz}, \cite{KuzPar} that
computed the Mellin transform of a certain functional of the stable
process in the form of a product of ratios of Alexeiewsky-Barnes
$G-$functions (or, equivalently, the double gamma function). In our
own work on the Selberg integral \cite{Me} we introduced a different
probability distribution having the same property that its Mellin
transform is given in the form of a finite product of ratios of
$G-$factors. The contribution of this paper is to show that there is
a whole family of Barnes beta distributions on the unit interval
$(0, 1]$ that extends this property to arbitrary multiple gamma
functions.

Third, we contribute to the probabilistic theory of the Selberg
integral complementing \cite{Yor}. We show that the Selberg integral
extends as a function of its dimension to the Mellin transform of a
probability distribution, which factorizes in terms of $\beta_{2,
2}^{-1}s.$ This leads us to a new interpretation of the Selberg
integral. %% as a transformation of $\beta_{1, 1}.$
%% can be interpreted a transformation of
%%into a product of $\beta_{2, 2}^{-1}$s, \emph{i.e.} a relationship
%%between $\beta_{1, 1}$ and $\beta_{2, 2}.$

%%and
%%explain how our results relate to the conjectured structure of the
%%limit lognormal stochastic measure.???

As an application of our results, we introduce a novel class of
power series, compute their Mellin transform, and prove their
positivity by relating them to the Laplace transform of the Barnes
beta distribution.

The main technical tool that we rely on in this paper is the
remarkable approach to multiple gamma functions due to Ruijsenaars
\cite{Ruij}. Ruijsenaars developed a novel Malmst\'en-type formula
for a class of functions that includes the multiple log-gamma
function as a special case. We prove the key infinitely divisibility
property in complete generality, that is for the whole Ruijsenaars
class, before specializing to the gamma functions. Most of our
proofs are elementary as the strength of his approach allows us to
reduce our arguments to simple properties of multiple Bernoulli
polynomials. %%In particular, we state and prove the general Shintani
%%identity \cite{KataOhts}, \cite{Shintani} in a new form.

The plan of the paper is as follows. In Section 1 we remind the
reader of the basic properties of the Barnes gamma functions
following \cite{mBarnes} and \cite{Ruij}. In Section 2 we state our
results. Section 3 gives examples of Barnes beta distributions.
Section 4 explains the connection between $\beta_{1,1},$
$\beta_{2,2},$ and the Selberg integral.
%%the known connection of our results with the Selberg integral and
%%conjectured connection with the limit
In Section 5 we present the proofs. Section 6
concludes with a summary.%% and remarks on some open questions. %%In the Appendix we
%%give a new derivation of the Shintani identity.

\section{Review of Multiple Gamma Functions}

Let $f(t)$ be of the Ruijsenaars class, \emph{i.e.} analytic for
$\Re(t)>0$ and at $t=0$ and of at worst polynomial growth as
$t\rightarrow \infty,$ confer \cite{Ruij}, Section 2. The main
example that corresponds to the case of Barnes multiple gamma
functions is
\begin{equation}\label{fdef}
f(t) = t^M \prod\limits_{j=1}^M (1-e^{-a_j t})^{-1}
\end{equation}
for some integer $M\geq 0$ and parameters $a_j>0,$ $j=1\cdots M.$
For concreteness, the reader can assume with little loss of
generality that $f(t)$ is defined by \eqref{fdef}. Slightly
modifying the definition in \cite{Ruij}, we define generalized
Bernoulli polynomials by
\begin{equation}\label{Bdefa}
B^{(f)}_{m}(x) \triangleq \frac{d^m}{dt^m}|_{t=0} \bigl[f(t)
e^{-xt}\bigr].
%%\sum\limits_{k=0}^m \binom{m}{k}\,f^{(k)}(0)
%%\,(-x)^{m-k}, \,\, m\in\mathbb{N}. %% , \;\text{or equivalently},\label{Bdefa} \\
%%f(t)\,e^{-xt} & = \sum\limits_{n=0}^\infty B^{(f)}_n(x)
%%\frac{t^n}{n!}, \label{Bdefb}
\end{equation}
%%for all sufficiently small $t.$
The generalized zeta function is defined by
\begin{equation}\label{zdef}
\zeta_M(s, \,w) \triangleq \frac{1}{\Gamma(s)} \int\limits_0^\infty
t^{s-1} e^{-wt}\,f(t) \,\frac{dt}{t^M}, \,\,\Re(s)>M,\,\Re(w)>0.
\end{equation}
It is shown in \cite{Ruij} that $\zeta_M(s, \,w)$ has an analytic
continuation to a function that is meromorphic in $s\in\mathbb{C}$
with simple poles at $s=1, 2, \cdots M.$ The generalized log-gamma
function is then defined by
\begin{equation}\label{Ldef}
L_M(w) \triangleq \partial_s \zeta_M(s, \,w)|_{s=0}, \,\,\Re(w)>0.
\end{equation}
It can be analytically continued to a function that is holomorphic
over $\mathbb{C}-(-\infty, 0].$ The key results of \cite{Ruij} that
we need are summarized in the following theorem.
\begin{theorem}[Ruijsenaars]\label{R}
$L_M(w)$ satisfies the Malmst\'en-type formula for $\Re(w)>0,$
\begin{equation}\label{key}
L_M(w) = \int\limits_0^\infty \frac{dt}{t^{M+1}} \Bigl(
e^{-wt}\,f(t) - \sum\limits_{k=0}^{M-1} \frac{t^k}{k!}\,B^{(f)}_k(w)
- \frac{t^M\,e^{-t}}{M!}\, B^{(f)}_M(w)\Bigr).
\end{equation}
$L_M(w)$ satisfies the asymptotic expansion,
\begin{gather}
L_M(w) = -\frac{1}{M!} B^{(f)}_M(w)\,\log(w) + \sum\limits_{k=0}^M
\frac{B^{(f)}_k(0) (-w)^{M-k}}{k!(M-k)!}\sum\limits_{l=1}^{M-k}
\frac{1}{l} + R_M(w), \label{asym}\\
R_M(w) = O(w^{-1}), \,|w|\rightarrow\infty, \, |\arg(w)|<\pi.
\label{asymremainder}
\end{gather}
%%The derivatives of $L_M(w)$ of all orders $k>M$ are given by
%%\begin{equation}\label{Gderiv}
%%\partial^k_w L_M(w) = (-1)^k (k-1)! \,Z_M(k, w).
%%\end{equation}
\end{theorem}

In the special case of the function $f(t)$ being defined by
\eqref{fdef}, the generalized zeta and gamma functions have
important additional properties. It is not difficult to show that
\eqref{zdef} becomes
\begin{equation}\label{mzdef}
\zeta_M\bigl(s,\,w\,|\,a\bigr) =
\sum\limits_{k_1,\cdots,k_M=0}^\infty \bigl(w+k_1 a_1+\cdots+k_M
a_M\bigr)^{-s},\,\,\Re(s)>M,\,\Re(w)>0,
\end{equation}
for $a=(a_1,\cdots, a_M),$ which is the formula given originally by
Barnes \cite{mBarnes} for the multiple zeta function. Let
$L_M(w\,|\,a)$ be defined by \eqref{Ldef} with $f(t)$ as in
\eqref{fdef}. Now, following \cite{Ruij}, define\footnote{Barnes
\cite{mBarnes} used a slightly different normalization, which does
not affect our results as we are primarily interested in ratios of
Barnes gamma functions.} the Barnes multiple gamma function by
\begin{equation}\label{mgamma}
\Gamma_M(w\,|\,a) \triangleq \exp\bigl(L_M(w\,|\,a)\bigr).
\end{equation}
It follows from \eqref{mzdef} and \eqref{mgamma} that
$\Gamma_M(w\,|\,a)$ satisfies the fundamental functional equation
\begin{equation}\label{feq}
\Gamma_{M}(w\,|\,a) =
\Gamma_{M-1}(w\,|\,\hat{a}_i)\,\Gamma_M\bigl(w+a_i\,|\,a\bigr),\,i=1\cdots
M, \,\,M=1, 2, 3\cdots,
\end{equation}
$\hat{a}_i = (a_1,\cdots, a_{i-1},\,a_{i+1},\cdots, a_{M}),$ and
$\Gamma_0(w) = 1/w,$ which is also due to \cite{mBarnes}.
%%It is understood in \eqref{feq} that the number of $a$'s matches the subscript.
By iterating \eqref{feq} one sees that $\Gamma_M(w\,|\,a)$ is
meromorphic over $\mathbb{C}$ having no zeroes and poles at
\begin{equation}\label{poles}
w=-(k_1 a_1+\cdots + k_M a_M),\; k_1\cdots k_M\in\mathbb{N},
\end{equation}
%%for all non-negative integers $k_1\cdots k_M.$
with multiplicity equal the number of $M-$tuples $(k_1, \cdots,
k_M)$ that satisfy \eqref{poles}.

We conclude our review of the Barnes functions by relating the
general results to the classical case of Euler's gamma and Hurwitz's
zeta functions. Following \cite{Ruij}, we have the identities
\begin{gather}
\Gamma_1(w\,|\,a) = \frac{a^{w/a-1/2}}{\sqrt{2\pi}} \,\Gamma(w/a), \label{gamma1}\\
\zeta_1(s, w\,|\,a) = a^{-s}\zeta(s,\,w/a),
\end{gather}
and the asymptotic expansion in Theorem 1.1 becomes Stirling's
series.

\section{Barnes Beta Distribution}
In this section we will define and describe the main properties of
what we call Barnes beta distributions.  We begin by introducing a
combinatorial operator $\mathcal{S}_N$ that plays a central role in
the formulation of our results.

Let $\lbrace b_k\rbrace,$ $k\in\mathbb{N}$ be a sequence of positive
real numbers and $N, M\in\mathbb{N}.$
%% CHANGE We assume $\lbrace b_k\rbrace$ to be fixed for the rest of the paper so that we can omit it from
%%the list of arguments for conciseness.
Let the symbol $\sum\limits_{k_1<\cdots<k_p=1}^N$ denote the sum
over all indices $k_i=1\cdots N,$ $i=1\cdots p,$ satisfying
$k_1<\cdots<k_p.$ Define the action of the operator $\mathcal{S}_N$
by
\begin{definition}\label{Soperator}
\begin{equation}\label{S}
(\mathcal{S}_Nf)(q\,|\,b) \triangleq \sum\limits_{p=0}^N (-1)^p
\sum\limits_{k_1<\cdots<k_p=1}^N
f\bigl(q+b_0+b_{k_1}+\cdots+b_{k_p}\bigr).
\end{equation}
\end{definition}
\noindent In other words, in \eqref{S} the action of $\mathcal{S}_N$
is defined as an alternating sum over all combinations of $p$
elements for every $p=0\cdots N.$ Given a function $f(t)$ of
Ruijsenaars class, confer Section 1, such that $f(t)>0$ for $t\geq
0,$ let $L_M(w)$ be the corresponding generalized log-gamma function
defined in \eqref{Ldef}. The main example is the function $f(t)$ in
\eqref{fdef} so that $L_M(w)=L_M(w\,|\,a)$ is the Barnes multiple
log-gamma function. We can now define the main object that we will
study in this paper.
\begin{definition}\label{bdef}
Given $q\in\mathbb{C}-(-\infty, -b_0],$ let
\begin{equation}\label{eta}
\eta_{M,N}(q\,|\,b) \triangleq \exp\Bigl(\bigl(\mathcal{S}_N
L_M\bigr)(q\,|\,b) - \bigl(\mathcal{S}_N L_M\bigr)(0\,|\,b)\Bigr).
\end{equation}
\end{definition}
\noindent The function $\eta_{M,N}(q\,|\,b)$ is holomorphic over
$q\in\mathbb{C}-(-\infty, -b_0]$ and equals a product of ratios of
generalized gamma functions by construction. Denoting
$\Gamma_M(w)=\exp\bigl(L_M(w)\bigr),$ it is easy to write out
examples of $\eta_{M,N}(q\,|\,b)$ for small $N.$
\begin{example}
\begin{gather}
\eta_{M, 0}(q\,|\,b) = \frac{\Gamma_M(q+b_0)}{\Gamma_M(b_0)}, \;
\eta_{M, 1}(q\,|\,b) = \frac{\Gamma_M(q+b_0)}{\Gamma_M(b_0)}
\frac{\Gamma_M(b_0+b_1)}{\Gamma_M(q+b_0+b_1)},
\\ \eta_{M, 2}(q\,|\,b)
= \frac{\Gamma_M(q+b_0)}{\Gamma_M(b_0)}
\frac{\Gamma_M(b_0+b_1)}{\Gamma_M(q+b_0+b_1)}
\frac{\Gamma_M(b_0+b_2)}{\Gamma_M(q+b_0+b_2)}
\frac{\Gamma_M(q+b_0+b_1+b_2)}{\Gamma_M(b_0+b_1+b_2)}.
\end{gather}
\end{example}

We now proceed to state our results.\footnote{Our results in the
case of $M=1$ correspond to a special case of the theory of Dufresne
distributions (also known as G distributions) and were first
obtained in \cite{Duf10}. The case of $M=N=2$ first appeared in
\cite{Me}.} We begin with the general case and then specialize to
that of multiple gamma functions.
\begin{theorem}[Existence]\label{main}
Given $M, N\in\mathbb{N}$ such that $M\leq N,$ the function
$\eta_{M,N}(q\,|\,b)$ is the Mellin transform of a probability
distribution on $(0, 1].$ Denote it by $\beta_{M, N}(b).$ Then,
\begin{equation}
{\bf E}\bigl[\beta_{M, N}(b)^q\bigr] = \eta_{M, N}(q\,|\,b),\;
\Re(q)>-b_0.
\end{equation}
The distribution $-\log\beta_{M, N}(b)$ is infinitely divisible on
$[0, \infty)$ and has the L\'evy-Khinchine decomposition
\begin{equation}\label{LKH}
{\bf E}\Bigl[\exp\bigl(-q\log\beta_{M, N}(b)\bigr)\Bigr] =
\exp\Bigl(\int\limits_0^\infty (e^{tq}-1) e^{-b_0
t}\prod\limits_{j=1}^N (1-e^{-b_j t}) \frac{f(t)}{t^{M+1}} dt\Bigr),
\; \Re(q)<b_0.
\end{equation}
\end{theorem}
\begin{corollary}[Structure]\label{Structure}
If $M=N,$ $\log\beta_{M,N}(b)$ is absolutely continuous. If $M<N,$
$-\log\beta_{M,N}(b)$ is compound Poisson and
\begin{equation}
{\bf P}\bigl[\beta_{M,N}(b)=1\bigr] =
\exp\Bigl(-\int\limits_0^\infty e^{-b_0 t}\prod\limits_{j=1}^N
(1-e^{-b_j t}) \frac{f(t)}{t^{M+1}} dt\Bigr).
\end{equation}
\end{corollary}
%%\begin{corollary}[Moment Problem]\label{momentproblem}
%%The moments of $\beta_{M,N}$ determine its law uniquely.
%%\end{corollary}
\begin{theorem}[Asymptotics]\label{Asymptotics}
If $M<N$ and $|\arg(q)|<\pi,$
\begin{equation}\label{ourasym}
\lim\limits_{q\rightarrow\infty} \eta_{M, N}(q\,|\,b) =
\exp\bigl(-(\mathcal{S}_N L_M)(0\,|\,b)\bigr).
\end{equation}
If $M=N$ and $|\arg(q)|<\pi,$
\begin{equation}\label{ourasymN}
\eta_{N, N}(q\,|\,b) = \exp\bigl(-b_1\cdots b_N f(0)\log(q) +
O(1)\bigr), \;q\rightarrow\infty.
\end{equation}
\end{theorem}
\begin{corollary}[Positivity]\label{Inequality} If $M<N,$
\begin{equation}
(\mathcal{S}_N L_M)(0\,|\,b) > 0.
\end{equation}
\end{corollary}

From now on we restrict our attention to Barnes multiple gamma
functions, \emph{i.e.} $f(t)$ is as in \eqref{fdef}, and write %%and $\beta_{M, N}(a, \,b)$
$\eta_{M, N}(q\,|\,a,\,b)$ to indicate dependence on $(a_1,\cdots,
a_M)$ and $(b_0,\cdots, b_N).$
%%We follow the same convention here as in \eqref{feq} that the number of $a$'s involved is the same as the (first) subscript $M.$
%% CHANGE
Also, $\hat{c}_i \triangleq (\cdots, c_{i-1},\,c_{i+1},\cdots)$ and
$\eta_{M, N}(q\,|\,a,\,b_j+x) \triangleq \eta_{M, N}(q\,|\,a,\cdots
b_{j-1}, b_j+x, b_{j+1}, \cdots),$ $x>0.$ Note that $\eta_{M,
N}(q\,|\,a,\,b)$ is symmetric in $(a_1,\cdots, a_M)$ and
$(b_1,\cdots, b_N).$ %%so that the reader can let $i=M$ and $j=N$ without loss of generality.
\begin{theorem}[Functional Equation]\label{FunctEquat}
$1\leq M\leq N,$ $q\in\mathbb{C}-(-\infty, -b_0],$ $i=1\cdots M,$
\begin{equation}
\eta_{M, N}(q+a_i\,|\,a,\,b) =
%%\eta_{M, N}(q\,|\,a,\,b) \frac{\eta_{M, N}(a_i\,|\,a,\,b)}{\eta_{M-1, N}(q\,|\,\hat{a}_i,\,b)}, \\
\eta_{M, N}(q\,|\,a,\,b)\,\exp\bigl(-(\mathcal{S}_N
L_{M-1})(q\,|\,\hat{a}_i, b)\bigr). \label{fe1}
\end{equation}
%%\eta_{M, N-1}(q\,|\,a,\,b_0+b_j, b_1,\cdots b_{j-1}, b_{j+1},\cdots b_N)
\end{theorem}
\begin{corollary}[Symmetries]\label{FunctSymmetry}
$1\leq M\leq N,$ $q\in\mathbb{C}-(-\infty, -b_0],$ $i=1\cdots M,$
$j=1\cdots N,$
\begin{align}
\eta_{M, N}(q\,|\,a,\,b_0+x)\,\eta_{M, N}(x\,|\,a,\,b) & = \eta_{M, N}(q+x\,|\,a,\,b), \label{fe2}\\
\eta_{M, N}(q\,|\,a,\,b)\,\eta_{M, N-1}(q\,|\,a,\,b_0+b_j,
\hat{b}_j) & =
%%\eta_{M, N-1}(q\,|\,a,\,\hat{b}_j)\frac{\eta_{M, N-1}(b_j\,|\,a,\,\hat{b}_j)}{\eta_{M, N-1}(q+b_j\,|\,a,\,\hat{b}_j)}, \; j=1\cdots N, \\
\eta_{M, N-1}(q\,|\,a,\,\hat{b}_j), \label{fe3}\\
\eta_{M, N}(q+a_i\,|\,a,\,b)\,\eta_{M-1, N}(q\,|\,\hat{a}_i,\,b) & =
\eta_{M, N}(q\,|\,a,\,b)\,\eta_{M, N}(a_i\,|\,a,\,b), \label{fe1eq}\\
\eta_{M, N}(q\,|\,a,\,b_j+a_i) \,\eta_{M-1,
N-1}(b_j\,|\,\hat{a}_i,\,\hat{b}_j) & = \eta_{M, N}(q\,|\,a,\,b)\,
%%\eta_{M-1, N-1}(q\,|\,\hat{a}_i,\,b_0+b_j, \hat{b}_j). \label{fe4}
\eta_{M-1, N-1}(q+b_j\,|\,\hat{a}_i,\,\hat{b}_j), \label{fe4} \\
\eta_{M, N}(q+a_i\,|\,a,\,b) \, \eta_{M-1,
N-1}(q\,|\,\hat{a}_i,\,\hat{b}_j) & = \eta_{M, N}(q\,|\,a,\,b) \,
\eta_{M-1, N-1}(q+b_j\,|\,\hat{a}_i,\,\hat{b}_j).
\label{funceqsymmetry}
\end{align}
\end{corollary}
\begin{corollary}[Moments]\label{moments}
Given $k\in\mathbb{N},$ the positive moments of $\beta_{M,
N}(a,\,b)^{a_i}$ satisfy
\begin{align}
{\bf E}\bigl[\beta_{M, N}(a, b)^{k a_i}\bigr] &  =
\exp\Bigl(-\sum\limits_{l=0}^{k-1} \bigl(\mathcal{S}_N
L_{M-1}\bigr)(l a_i\,|\,\hat{a}_i, b)\Bigr), \nonumber \\
& = \prod\limits_{l=0}^{k-1} \Bigl[\frac{\prod\limits_{j_1=1}^N
\Gamma_{M-1}(l a_i+b_0+b_{j_1}\,|\,\hat{a}_i)}{\Gamma_{M-1}(l
a_i+b_0\,|\,\hat{a}_i)} \star \nonumber \\
& \star \frac{\prod\limits_{j_1<j_2<j_3}^N \Gamma_{M-1}(l
a_i+b_0+b_{j_1}+b_{j_2}+b_{j_3}\,|\,\hat{a}_i)}{\prod\limits_{j_1<j_2}^N
\Gamma_{M-1}(l a_i+b_0+b_{j_1}+b_{j_2}\,|\,\hat{a}_i)}\cdots\Bigr].
\label{posmom}
\end{align}
Given $k\in\mathbb{N}$ such that $k a_i<b_0,$ the negative moments
of $\beta_{M, N}(a, b)^{a_i}$ satisfy
\begin{align}
{\bf E}\bigl[\beta_{M, N}(a, b)^{-k a_i}\bigr] & =
\exp\Bigl(\sum\limits_{l=0}^{k-1} \bigl(\mathcal{S}_N
L_{M-1}\bigr)(-(l+1) a_i\,|\,\hat{a}_i, b)\Bigr), \nonumber \\
& = \prod\limits_{l=0}^{k-1} \Bigl[\frac{\Gamma_{M-1}(-(l+1)
a_i+b_0\,|\,\hat{a}_i)}{\prod\limits_{j_1=1}^N \Gamma_{M-1}(-(l+1)
a_i+b_0+b_{j_1}\,|\,\hat{a}_i)} \star \nonumber
\\ & \star\frac{\prod\limits_{j_1<j_2}^N \Gamma_{M-1}(-(l+1)
a_i+b_0+b_{j_1}+b_{j_2}\,|\,\hat{a}_i)}
{\prod\limits_{j_1<j_2<j_3}^N \Gamma_{M-1}(-(l+1)
a_i+b_0+b_{j_1}+b_{j_2}+b_{j_3}\,|\,\hat{a}_i)}\cdots\Bigr].
\label{negmom}
\end{align}
\end{corollary}
\begin{corollary}[Laplace Transform]\label{laplace}
The power series
\begin{equation}\label{powerseries}
\mathcal{L}^{(i)}_{M,N}(x\,|\,a, b) \triangleq
\sum\limits_{k=0}^\infty \frac{(-x)^k}{k!}
\exp\Bigl(-\sum\limits_{l=0}^{k-1} \bigl(\mathcal{S}_N
L_{M-1}\bigr)(l a_i\,|\,\hat{a}_i, b)\Bigr)
\end{equation}
has infinite radius of convergence and gives the Laplace transform
of $\beta_{M, N}(a, b)^{a_i}.$
\begin{equation}
{\bf E}\Bigl[\exp\bigl(-x \beta_{M, N}(a, b)^{a_i}\bigr)\Bigr] =
\mathcal{L}^{(i)}_{M,N}(x\,|\,a, b),\; x>0.
\end{equation}
In particular, for $x>0,$
\begin{equation}
\mathcal{L}^{(i)}_{M,N}(x\,|\,a, b)>0.
\end{equation}
\end{corollary}
\begin{corollary}[Ramanujan Representation]\label{Rama}
The Laplace transform satisfies
\begin{equation}
\int\limits_0^\infty x^{q-1}\mathcal{L}^{(i)}_{M,N}(x\,|\,a, b) \,dx
= \Gamma(q)\, \eta_{M,N}(-q a_i\,|\,a, b),\; 0<\Re(q)<b_0/a_i.
\end{equation}
\end{corollary}
\begin{theorem}[Shintani Factorization]\label{Factorization}
Given $1\leq M\leq N$ and $q\in\mathbb{C}-(-\infty, -b_0],$
\begin{align}
\eta_{M,N}(q\,|\,a, b) & = \prod\limits_{k=0}^\infty
\eta_{M-1,N}(q\,|\,\hat{a}_i, b_0+ka_i), \label{infinprod1}\\
& = \prod\limits_{k=0}^\infty \frac{\eta_{M-1,N}(q+k
a_i\,|\,\hat{a}_i, b)}{\eta_{M-1,N}(k
a_i\,|\,\hat{a}_i, b)}, \label{infinprod2} \\
& = \prod\limits_{k=0}^\infty \frac{\eta_{M-1,N-1}(q+k
a_i\,|\,\hat{a}_i, \hat{b}_j)}{\eta_{M-1,N-1}(k a_i\,|\,\hat{a}_i,
\hat{b}_j)} \, \frac{\eta_{M-1,N-1}(k a_i+b_j\,|\,\hat{a}_i,
\hat{b}_j)}{\eta_{M-1,N-1}(q+k a_i+b_j\,|\,\hat{a}_i, \hat{b}_j)}.
\label{infinprod3}
\end{align}
\end{theorem}
\begin{remark}
This factorization is the analogue of the factorization of the
Barnes multiple gamma function $\Gamma_M(w\,|\,a)$ into the product
of ratios of $\Gamma_{M-1}(w\,|\,\hat{a}_i)$ originally due to
\cite{Shintani} for $M=2$ and, in general, due to \cite{KataOhts}.
\end{remark}
\begin{corollary}[Solution to Functional Equations]\label{solution}
The infinite product representation in Theorem \ref{Factorization}
is the solution to the functional equation in Theorem
\ref{FunctEquat}.
%%\eqref{fe1} and the representation in
%%\eqref{infinprod3} is the solution to the functional equation
%%\eqref{fe4}
\end{corollary}

We conclude this section with two results that hold for special
values of $a$ and $b.$
\begin{theorem}[Reduction to Independent Factors]\label{Reduction}
Given $i, j,$ if $b_j=n\,a_i$ for $n\in\mathbb{N},$
%%$\beta_{M, N}(a, b)$ decomposes
\begin{equation}\label{reductequat}
\beta_{M, N}(a, b) \overset{{\rm in \,law}}{=} \prod_{k=0}^{n-1}
\beta_{M-1, N-1}\bigl(\hat{a}_i, b_0+ka_i,\, \hat{b}_j\bigr).
\end{equation}
\end{theorem}
\begin{theorem}[Moments]\label{Momentsaspecial} Let $a_i=1$ for all $i=1\cdots M.$ Then, for any
$n\in\mathbb{N},$
\begin{align}
{\bf E}\bigl[\beta_{M, N}(a, b)^n\bigr]  & =
\prod\limits_{i=1}^{M-1}e^{(-1)^i \binom{n}{i} (\mathcal{S}_N
L_{M-i})(0\,|\,b)} \prod\limits_{i_1=0}^{n-1}
\prod\limits_{i_2=0}^{i_1-1} \cdots \prod\limits_{i_M=0}^{i_{M-1}-1}
e^{(-1)^M (\mathcal{S}_N L_{0})(i_M\,|\,b)}, \nonumber \\
& = \prod\limits_{i=1}^{M-1}e^{(-1)^i \binom{n}{i} (\mathcal{S}_N
L_{M-i})(0\,|\,b)} \prod\limits_{i_1=0}^{n-1}
\prod\limits_{i_2=0}^{i_1-1} \cdots \prod\limits_{i_M=0}^{i_{M-1}-1}
\Bigl[\frac{\prod\limits_{j_1=1}^N (i_M+b_0+b_{j_1})}{(i_M+b_0)} \star \nonumber \\
& \star \frac{\prod\limits_{j_1<j_2<j_3}^N
(i_M+b_0+b_{j_1}+b_{j_2}+b_{j_3})}{\prod\limits_{j_1<j_2}^N
(i_M+b_0+b_{j_1}+b_{j_2})}\cdots\Bigr]^{(-1)^M}.
\end{align}
\end{theorem}

We note that the structure of $\beta_{M,N}(a)$ depends on
rationality of $(a_1,\cdots, a_M)$ and $(b_1,\cdots, b_M).$ This is
clear from Definition \ref{bdef} as this structure is determined by
ratios of multiple gamma functions that have poles specified in
\eqref{poles}. This phenomenon was studied in a different context
for the double gamma function in \cite{HubKuz} and \cite{Kuz}.

\section{Examples}

It is not difficult to compute $\beta_{M,N}(a, b)$ for small $M$ and
$N.$ We give four examples that can be checked by direct inspection.
\begin{example}\label{example} Let $\delta(x-1)$ be shorthand for an atom at $x=1.$
\begin{align}
\beta_{0, 0} & = b_0\,x^{b_0-1}\,dx, \\
\beta_{0, 1} & = \frac{b_1}{b_0+b_1} b_0\,x^{b_0-1}\,dx +
\frac{b_0}{b_0+b_1}\delta(x-1)\,dx, \\
\beta_{1,1} & =
a_1\frac{\Gamma\bigl((b_0+b_1)/a_1\bigr)}{\Gamma\bigl(b_0/a_1\bigr)\Gamma\bigl(b_1/a_1\bigr)}
x^{b_0-1}(1-x^{a_1})^{b_1/a_1-1}\,dx, \label{beta}\\
\beta_{0, 2} & =
\frac{b_0b_1b_2(b_0+b_1+b_2)}{(b_0+b_1)(b_0+b_2)(b_1+b_2)}
x^{b_0-1}(1-x^{b_1+b_2})\,dx + \delta(x-1)\frac{b_0
(b_0+b_1+b_2)}{(b_0+b_1)(b_0+b_2)}\,dx.
\end{align}
\end{example}

In the rest of this section we will focus on %%the quite non-trivial example of $M=N=2.$
the special case of $M=N=2$ in order to illustrate the general
theory with a concrete yet quite non-trivial example. In addition,
this case is also of a particular interest in the probabilistic
theory of the Selberg integral that we will review in Section 4. Let
$a_1=1$ and $a_2=\tau>0$ and write $\beta_{2, 2}(\tau, b),$
$\eta_{2,2}(q\,|\,\tau, b),$ and
$\Gamma_2\bigl(w\,|\,(1,\tau)\bigr)=\Gamma_2(w\,|\,\tau)$ for
brevity. From Definition \ref{bdef} and Theorem \ref{main} we have
${\bf E}\bigl[\beta_{2, 2}(\tau, b)^q\bigr] = \eta_{2,2}(q\,|\,\tau,
b)$ for $\Re(q)>-b_0$ and
\begin{equation}
\eta_{2,2}(q\,|\,\tau, b) =
\frac{\Gamma_2(q+b_0\,|\,\tau)}{\Gamma_2(b_0\,|\,\tau)}
\frac{\Gamma_2(b_0+b_1\,|\,\tau)}{\Gamma_2(q+b_0+b_1\,|\,\tau)}
\frac{\Gamma_2(b_0+b_2\,|\,\tau)}{\Gamma_2(q+b_0+b_2\,|\,\tau)}
\frac{\Gamma_2(q+b_0+b_1+b_2\,|\,\tau)}{\Gamma_2(b_0+b_1+b_2\,|\,\tau)}.
\end{equation}
The asymptotic behavior of $\eta_{2,2}(q\,|\,\tau, b)$ %%in the limit
follows from Theorem \ref{Asymptotics}.
\begin{equation}
\eta_{2, 2}(q\,|\,\tau, b) = \exp\Bigl(-\frac{b_1 b_2}{\tau}\log(q)
+ O(1)\Bigr), \; q\rightarrow\infty,\;|\arg(q)|<\pi.
\end{equation}
Using \eqref{gamma1}, the functional equation in Theorem
\ref{FunctEquat} takes the form
\begin{align}
\eta_{2, 2}(q+1\,|\,\tau, b) & = \eta_{2, 2}(q\,|\,\tau, b)\,
\frac{\Gamma\bigl((q+b_0+b_1)/\tau\bigr)\Gamma\bigl((q+b_0+b_2)/\tau\bigr)}
{\Gamma\bigl((q+b_0)/\tau\bigr)\Gamma\bigl((q+b_0+b_1+b_2)/\tau\bigr)},
\\
\eta_{2, 2}(q+\tau\,|\,\tau, b) & = \eta_{2, 2}(q\,|\,\tau, b)\,
\frac{\Gamma(q+b_0+b_1)\Gamma(q+b_0+b_2)}
{\Gamma(q+b_0)\Gamma\bigl(q+b_0+b_1+b_2)}.
\end{align}
The positive moments in Corollary \ref{moments} for $k\in\mathbb{N}$
are
\begin{align}
{\bf E}\bigl[\beta_{2, 2}(\tau, b)^k\bigr] & =
\prod\limits_{l=0}^{k-1}
\Bigl[\frac{\Gamma\bigl((l+b_0+b_1)/\tau\bigr)\,\Gamma\bigl((l+b_0+b_2)/\tau\bigr)}{\Gamma\bigl((l
+b_0)/\tau\bigr)\,\Gamma\bigl((l+b_0+b_1+b_2)/\tau\bigr)}\Bigr], \\
{\bf E}\bigl[\beta_{2, 2}(\tau, b)^{k\tau}\bigr] & =
\prod\limits_{l=0}^{k-1}
\Bigl[\frac{\Gamma(l\tau+b_0+b_1)\,\Gamma(l\tau+b_0+b_2)}{\Gamma(l\tau
+b_0)\,\Gamma(l\tau+b_0+b_1+b_2)}\Bigr].
\end{align}
The negative moments are
\begin{align}
{\bf E}\bigl[\beta_{2, 2}(\tau, b)^{-k}\bigr] & =
\prod\limits_{l=0}^{k-1} \Bigl[\frac{\Gamma\bigl((-(l+1)
+b_0)/\tau\bigr)\,\Gamma\bigl((-(l+1)+b_0+b_1+b_2)/\tau\bigr)}{\Gamma\bigl((-(l+1)+b_0+b_1)/\tau\bigr)\,
\Gamma\bigl((-(l+1)+b_0+b_2)/\tau\bigr)}\Bigr], \; k<b_0,\\
{\bf E}\bigl[\beta_{2, 2}(\tau, b)^{-k\tau}\bigr] & =
\prod\limits_{l=0}^{k-1} \Bigl[\frac{\Gamma(-(l+1)\tau
+b_0)\,\Gamma(-(l+1)\tau+b_0+b_1+b_2)}{\Gamma((-(l+1)\tau+b_0+b_1)\,\Gamma(-(l+1)\tau+b_0+b_2)}\Bigr],\;k\tau<b_0.
\end{align}
The positivity conditions in Corollary \ref{laplace} for $x>0$ are
\begin{align}
{\bf E}\Bigl[\exp\bigl(-x \beta_{2, 2}(\tau, b)\bigr)\Bigr] & =
\sum\limits_{k=0}^\infty \frac{(-x)^k}{k!} \prod\limits_{l=0}^{k-1}
\Bigl[\frac{\Gamma\bigl((l+b_0+b_1)/\tau\bigr)\,\Gamma\bigl((l+b_0+b_2)/\tau\bigr)}{\Gamma\bigl((l
+b_0)/\tau\bigr)\,\Gamma\bigl((l+b_0+b_1+b_2)/\tau\bigr)}\Bigr]>0, \\
{\bf E}\Bigl[\exp\bigl(-x \beta_{2, 2}(\tau, b)^\tau\bigr)\Bigr] & =
\sum\limits_{k=0}^\infty \frac{(-x)^k}{k!}
\prod\limits_{l=0}^{k-1}\Bigl[\frac{\Gamma(l\tau+b_0+b_1)\,\Gamma(l\tau+b_0+b_2)}
{\Gamma(l\tau +b_0)\,\Gamma(l\tau+b_0+b_1+b_2)}\Bigr]>0.
\end{align}
Corollary \ref{Rama} gives for $0<\Re(q)<b_0$ and
$0<\Re(q)<b_0/\tau,$ respectively,
\begin{align}
\Gamma(q)\, \eta_{2,2}(-q\,|\,\tau, b) & = \int\limits_0^\infty
x^{q-1}\Bigl\{ \sum\limits_{k=0}^\infty \frac{(-x)^k}{k!}
\prod\limits_{l=0}^{k-1}
\Bigl[\frac{\Gamma\bigl((l+b_0+b_1)/\tau\bigr)\,\Gamma\bigl((l+b_0+b_2)/\tau\bigr)}{\Gamma\bigl((l
+b_0)/\tau\bigr)\,\Gamma\bigl((l+b_0+b_1+b_2)/\tau\bigr)}\Bigr]\Bigr\} \,dx, \\
\Gamma(q)\, \eta_{2,2}(-q\tau\,|\,\tau, b) & = \int\limits_0^\infty
x^{q-1} \Bigl\{\sum\limits_{k=0}^\infty \frac{(-x)^k}{k!}
\prod\limits_{l=0}^{k-1}\Bigl[\frac{\Gamma(l\tau+b_0+b_1)\,\Gamma(l\tau+b_0+b_2)}
{\Gamma(l\tau +b_0)\,\Gamma(l\tau+b_0+b_1+b_2)}\Bigr]\Bigr\} \,dx.
\end{align}
Finally, the factorization equations in Theorem \ref{Factorization}
are
\begin{align} %%
\eta_{2,2}(q\,|\,\tau, b) &  =
 \prod\limits_{k=0}^\infty\Bigl[
\frac{\Gamma((q+k+b_0)/\tau) }{\Gamma((k+b_0)/\tau)}
\frac{\Gamma((k+b_0+b_1)/\tau)}{\Gamma((q+k+b_0+b_1)/\tau)}
\frac{\Gamma((k+b_0+b_2)/\tau)}{\Gamma((q+k+b_0+b_2)/\tau)} \star
\nonumber \\ & \star
\frac{\Gamma((q+k+b_0+b_1+b_2)/\tau)}{\Gamma((k+b_0+b_1+b_2)/\tau)}\Bigr],
\\
\eta_{2,2}(q\,|\,\tau, b) & = \prod\limits_{k=0}^\infty\Bigl[
\frac{\Gamma(q+k\tau+b_0)}{\Gamma(k\tau+b_0)}
\frac{\Gamma(k\tau+b_0+b_1)}{\Gamma(q+k\tau+b_0+b_1)}
\frac{\Gamma(k\tau+b_0+b_2)}{\Gamma(q+k\tau+b_0+b_2)} \star
\nonumber \\ & \star
\frac{\Gamma(q+k\tau+b_0+b_1+b_2)}{\Gamma(k\tau+b_0+b_1+b_2)}\Bigr].
\end{align}

The density of $\beta_{2, 2}(\tau, b)$ can be computed by Laplace
transform inversion. This computation requires a separate study
similar to \cite{HubKuz} as the structure of the residues of
$\eta_{2, 2}(q\,|\,\tau, b)$ depends on the rationality of $\tau,$
$b_1,$ and $b_2.$ %%We leave this problem for further research.
%% much in the same way as we computed the distributions in Example \eqref{example}.

\section{$\beta_{2, 2}(\tau, b)$ and Selberg Integral}

In this section we will review the application of $\beta_{2,
2}(\tau, b)$ to the probabilistic structure of the celebrated
Selberg
integral that we developed in \cite{Me} using %%a different approach.
special properties of the Alexeiewsky-Barnes $G-$function (confer
the appendix of \cite{Me} for a review of the $G-$function). The
goal of re-formulating this structure here is to put it into the
general framework of the Barnes beta distributions, which leads to a
new interpretation of the Selberg integral.
%% and give a non-trivial
%%application of $\beta_{2, 2}(\tau).$
%% and that served as the primary motivating factor for studying Barnes beta distributions in
%%general.
%%The results in this section were originally derived using
%%
%%

The starting point of the probabilistic study of the Selberg
integral is the following remarkable formula due to Selberg
\cite{Selberg}. Given $0<\mu<2,$ $\lambda_i>-\mu/2,$ and $1\leq l<
2/\mu,$
\begin{align}
S_{\mu,l}\bigl[s^{\lambda_1}(1-s)^{\lambda_2}\bigr] & =
\prod_{k=0}^{l-1} \frac{\Gamma(1-(k+1)\mu/2)}{\Gamma(1-\mu/2)}
\frac{\Gamma(1+\lambda_1-k\mu/2)\Gamma(1+\lambda_2-k\mu/2)}
{\Gamma(2+\lambda_1+\lambda_2-(l+k-1)\mu/2)}, \label{Selberg} \\
S_{\mu,l}[\varphi] & \triangleq \int\limits_{[0,\,1]^l}
\prod_{i=1}^l \varphi(s_i)\, \prod\limits_{i<j}^l |s_i-s_j|^{-\mu}
ds_1\cdots ds_l. \label{Soperatordef}
\end{align}
%%confer
%and \cite{ForresterWarnaar}
The reader who is familiar with the classical approach to the
Selberg integral, confer \cite{ForresterBook}, will notice that we
have written Selberg's formula in a somewhat peculiar form. The
reason for restricting $0<\mu<2$ is that in this case, given a
general function $\varphi(s),$ $S_{\mu,l}[\varphi]$ equals the $l$th
moment of the probability distribution constructed by integrating
$\varphi(s)$ with respect to the limit lognormal stochastic
measure,\footnote{We will not attempt to quantify this statement
here as it would take us too far afield and it is solely used to
motivate Theorem \ref{BSM} and our interpretation of the Selberg
integral in Remark \ref{interptet}.} confer \cite{Me}. In an attempt
to compute this distribution for
$\varphi(s)=s^{\lambda_1}(1-s)^{\lambda_2},$ in \cite{Me} we
constructed\footnote{In the special case of $\lambda_1=\lambda_2=0$
an equivalent formula for the Mellin transform first appeared in
\cite{Me4}. The general case was first considered by \cite{FLDR},
who gave an equivalent expression for the right-hand side of
\eqref{M} and so matched the moments without proving that it
corresponds to a probability distribution.} and factorized a
probability distribution\footnote{If $M_{(\mu, \lambda_1,
\lambda_2)}$ is the sought distribution as conjectured in \cite{Me},
then the equality of their moments gives a probabilistic derivation
of Selberg's formula by \eqref{Soperatordef} and Theorem \ref{BSM}.}
having the moments given by Selberg's formula in \eqref{Selberg}.
\begin{theorem}\label{BSM}
Let $0<\mu<2,$ $\lambda_i>-\mu/2,$ and $\tau=2/\mu.$ Define the
Mellin transform
\begin{align}
{\bf E}\bigl[M^q_{(\mu, \lambda_1, \lambda_2)}\bigr] & \triangleq
\tau^{\frac{q}{\tau}} (2\pi)^{q}\,\Gamma^{-q}\bigl(1-1/\tau\bigr)
\frac{\Gamma_2(1-q+\tau(1+\lambda_1)\,|\,\tau)}{\Gamma_2(1+\tau(1+\lambda_1)\,|\,\tau)}\star
\nonumber \\ & \star
\frac{\Gamma_2(1-q+\tau(1+\lambda_2)\,|\,\tau)}{\Gamma_2(1+\tau(1+\lambda_2)\,|\,\tau)}
\frac{\Gamma_2(-q+\tau\,|\,\tau)}{\Gamma_2(\tau\,|\,\tau)}
\frac{\Gamma_2(2-q+\tau(2+\lambda_1+\lambda_2)\,|\,\tau)}{\Gamma_2(2-2q+\tau(2+\lambda_1+\lambda_2)\,|\,\tau)}
\label{M}
\end{align}
for $\Re(q)<\tau.$ Then, $M_{(\mu, \lambda_1, \lambda_2)}$ is a
probability distribution on $(0,\infty)$ and %%its moments satisfy \eqref{Selberg}.
\begin{equation}
{\bf E}\bigl[M^l_{(\mu, \lambda_1, \lambda_2)}\bigr] =
S_{\mu,l}\bigl[s^{\lambda_1}(1-s)^{\lambda_2}\bigr],\; 1\leq l <
\tau.
\end{equation}
$\log M_{(\mu, \lambda_1, \lambda_2)}$ is absolutely continuous and
infinitely divisible.
%%an absolutely continuous
%%having infinitely divisible logarithm
%%such that ${\bf E}\bigl[M^l_{(\mu, \lambda_1, \lambda_2)}\bigr] =
\end{theorem}

We can now relate the Selberg integral and $\beta_{2,2}(\tau, b).$
Let $\tau>1$ and define
\begin{equation}
L \triangleq \exp\bigl(\mathcal{N}(0,\,4\log 2/\tau)\bigr), \; Y
\triangleq \tau\,y^{-1-\tau}\exp\bigl(-y^{-\tau}\bigr)\,dy,\; y>0,
\end{equation}
\emph{i.e.} $\log L$ is a zero-mean normal with variance $4\log
2/\tau$ and $Y$ is a power of the exponential. Given
$\lambda_i>-1/\tau,$ let $X_1,\,X_2,\,X_3$ have the $\beta^{-1}_{2,
2}(\tau, b)$ distribution with the parameters\footnote{The
parameters of $X_1$
%%do not satisfy $b_1, b_2>0$ as we required all along. They do
satisfy $b_1 b_2>0,$ which is sufficient for Theorem \ref{main} to
hold by Theorem 3.5 in \cite{Me}.}
\begin{align}
X_1 &\triangleq \beta_{2,2}^{-1}\Bigl(\tau,
b_0=1+\tau+\tau\lambda_1,\,b_1=\tau(\lambda_2-\lambda_1)/2, \,
b_2=\tau(\lambda_2-\lambda_1)/2\Bigr),\\
X_2 & \triangleq \beta_{2,2}^{-1}\Bigl(\tau,
b_0=1+\tau+\tau(\lambda_1+\lambda_2)/2,\,b_1=1/2,\,b_2=\tau/2\Bigr),\\
X_3 & \triangleq \beta_{2,2}^{-1}\Bigl(\tau, b_0=1+\tau,\,
b_1=(1+\tau+\tau\lambda_1+\tau\lambda_2)/2, \,
b_2=(1+\tau+\tau\lambda_1+\tau\lambda_2)/2\Bigr).
\end{align}
%%thereby intrinsically relating the Selberg integral and $\beta_{2,2}(\tau).$
\begin{theorem}\label{B2} Let $\tau=2/\mu.$ $M_{(\mu, \lambda_1, \lambda_2)}$
decomposes into independent factors,
\begin{equation}\label{Decomposition}
M_{(\mu, \lambda_1, \lambda_2)} \overset{{\rm in \,law}}{=} 2\pi\,
2^{-\bigl[3(1+\tau)+2\tau(\lambda_1+\lambda_2)\bigr]/\tau}\,\Gamma\bigl(1-1/\tau\bigr)^{-1}\,
L\,X_1\,X_2\,X_3\,Y.
\end{equation}
\end{theorem}
\begin{remark}[Interpretation of Selberg Integral]\label{interptet}
The function $s^{\lambda_1} (1-s)^{\lambda_2}$ in \eqref{Selberg}
is, up to a constant, the density of $\beta_{1, 1}\bigl(a_1=1,
b_0=1+\lambda_1, b_1=1+\lambda_2\bigr),$
confer \eqref{beta}. %%Selberg's formula and Theorems \ref{BSM}, \ref{B2} imply
Selberg's formula and Theorems \ref{BSM} and \ref{B2} extend the
integral $S_{\mu, l}\bigl[\text{pdf of} \,\beta_{1,1}\bigr],$ viewed
as a function of $l,$ to the Mellin transform of $const
\,L\,X_1\,X_2\,X_3\,Y.$
\begin{equation}
\text{pdf of} \,\beta_{1,1} \xrightarrow{\text{Selberg}} S_{\mu,
l}\bigl[\text{pdf of} \,\beta_{1,1}\bigr] \xrightarrow{\text{Ths.
\ref{BSM}, \ref{B2}}} const \,L\,X_1\,X_2\,X_3\,Y,
\end{equation}
%%Theorem \ref{BSM} says that the Selberg integral transforms it into
%%extends analytically as a function of its dimension to
%%the Mellin transform of
\emph{i.e.} the Selberg integral can be interpreted
probabilistically as a transformation of $\beta_{1,1}$ into the
product in \eqref{Decomposition}. It is an open question how to
extend this mechanism to $\beta_{M, M},$ \emph{i.e.} how to compute
a probability distributions having $S_{\mu, l}\bigl[\text{pdf of}
\,\beta_{M,M}\bigr]$ as its moments.
\end{remark}

\section{Proofs}
In this section we will give proofs of the results in Section 2. The
proofs rely on Theorem 1.1, properties of infinitely divisible
distributions, and the following lemma.
\begin{lemma}[Main Lemma]\label{mylemma}
Let $f(t)$ be of the Ruijsenaars class and the generalized Bernoulli
polynomials be defined by \eqref{Bdefa}. Let $\lbrace b_k\rbrace,$
$k\in\mathbb{N},$ be a sequence of real
numbers, %% let the operator $\mathcal{S}_N$ be defined by \eqref{S}.
$n, r\in\mathbb{N},$ and $q\in\mathbb{C}.$ Define the function
$g(t)$
\begin{equation}\label{gfunction}
g(t) \triangleq f(t) e^{-qt} \frac{d^r}{dt^r}\bigl[e^{-b_0
t}\prod\limits_{j=1}^N (1-e^{-b_j t})\bigr].
\end{equation}
Then,
\begin{align}
g^{(n)}(0) & = \sum\limits_{m=0}^n \binom{n}{m} B^{(f)}_{n-m}(q)
\frac{d^{m+r}}{dt^{m+r}}|_{t=0}\bigl[e^{-b_0 t}\prod\limits_{j=1}^N
(1-e^{-b_j t})\bigr], \label{line1} \\
& = (-1)^r \sum\limits_{p=0}^N (-1)^p
\sum\limits_{k_1<\cdots<k_p=1}^N \bigl(b_0+\sum b_{k_j}\bigr)^r \,
B^{(f)}_n\bigl(q+b_0+\sum b_{k_j}\bigr), \label{line2} \\
&= 0,\;{\rm if}\;r+n<N, \label{deriv1}\\
&= f(0)\,N! \prod\limits_{j=1}^N b_j,\;{\rm
if}\;r+n=N.\label{deriv2}
\end{align}
\end{lemma}
\begin{proof}
The expression in \eqref{line1} follows from \eqref{Bdefa}. Using
the identity
\begin{equation}
\prod\limits_{j=1}^N (1-e^{-b_j t}) = \sum\limits_{p=0}^N (-1)^p
\sum\limits_{k_1<\cdots<k_p=1}^N
\exp\bigl(-(b_{k_1}+\cdots+b_{k_p})t\bigr),
\end{equation}
we can write
\begin{equation}\label{exponID}
\frac{d^r}{dt^r}\bigl[e^{-b_0 t}\prod\limits_{j=1}^N (1-e^{-b_j
t})\bigr] = (-1)^r \sum\limits_{p=0}^N (-1)^p
\sum\limits_{k_1<\cdots<k_p=1}^N \bigl(b_0+\sum b_{k_j}\bigr)^r
\exp\bigl(-(b_0+\sum b_{k_j})t\bigr).
\end{equation}
Substituting this expression into \eqref{gfunction} and recalling
\eqref{Bdefa}, we obtain \eqref{line2}. \eqref{deriv1} is immediate
from the definition of $g(t)$ in \eqref{gfunction}, and
\eqref{deriv2} follows from \eqref{line1}.
\end{proof}
\begin{corollary}\label{mainauxID}
\begin{align}
\bigl(\mathcal{S}_N\,B^{(f)}_n\bigr)(q)& =0,\; n=0\cdots N-1,
\label{mainID} \\
\bigl(\mathcal{S}_N\,B^{(f)}_N\bigr)(q)&=f(0) \,N!
\prod\limits_{j=1}^N b_j, \label{mainIDN} \\
\bigl(\mathcal{S}_N \,x^n\bigr)(q)&=0,\;n=0\cdots N-1, \label{auxID}\\
\bigl(\mathcal{S}_N \,x^N\bigr)(q)&=(-1)^N N!\prod\limits_{j=1}^N
b_j. \label{auxIDN}
\end{align}
\end{corollary}
\begin{proof}
\eqref{mainID} and \eqref{mainIDN} follow from Lemma \ref{mylemma}
by setting $r=0$ and recalling \eqref{S}. \eqref{auxID} and
\eqref{auxIDN} follows from \eqref{mainID} and \eqref{mainIDN} by
letting $f(t)=1$ in Lemma \ref{mylemma} so that the corresponding
Bernoulli polynomials are $B^{(f)}_n(x) = (-x)^n.$
\end{proof}
\begin{proof}[Proof of Theorem \ref{main}]
Let $M\leq N$ and $\Re(q)>-b_0.$ We start with Definition \ref{bdef}
and substitute \eqref{key} for $L_M(w).$ By \eqref{mainID} in
Corollary \ref{mainauxID} and linearity of $\mathcal{S}_N,$ we
obtain
\begin{equation}\label{inter}
\eta_{M,N}(q\,|\,b) = \exp\Bigl(\int\limits_0^\infty
\Bigl[\bigl(\mathcal{S}_N \,\exp(-xt)\bigr)(q)-\bigl(\mathcal{S}_N
\,\exp(-xt)\bigr)(0)\Bigr]f(t) dt/t^{M+1}\Bigr).
\end{equation}
Letting $r=0$ in \eqref{exponID}, we have the identity
\begin{equation}
e^{-b_0 t} e^{-qt}\prod\limits_{j=1}^N (1-e^{-b_j t}) =
\bigl(\mathcal{S}_N \exp(-xt)\bigr)(q)
\end{equation}
so that \eqref{inter} can be simplified to
\begin{equation}
\eta_{M,N}(q\,|\,b) = \exp\Bigl(\int\limits_0^\infty
(e^{-qt}-1)e^{-b_0t}\prod\limits_{j=1}^N (1-e^{-b_j
t})f(t)\,dt/t^{M+1}\Bigr).
\end{equation}
\noindent This is the canonical representation of the Laplace
transform of an infinitely divisible distribution on $[0,\,\infty),$
confer Theorem 4.3 in Chapter 3 of \cite{SteVHar}.
\begin{align}
\eta_{M,N}(q\,|\,b) & = \exp\Bigl(-\int\limits_0^\infty
(1-e^{-tq})dK^{(f)}_{M,N}(t\,|\,b)/t\Bigr), \label{Keq}\\
dK^{(f)}_{M,N}(t\,|\,b) & \triangleq e^{-b_0t} \prod\limits_{j=1}^N
(1-e^{-b_j t}) f(t)\,dt/t^M. \label{Kdef}
\end{align}
It remains to note that $dK^{(f)}_{M,N}(t\,|\,b)$ satisfies the
required integrability condition
\begin{equation}
\int\limits_0^\infty e^{-st} dK^{(f)}_{M,N}(t\,|\,b) =
\int\limits_0^\infty e^{-st} e^{-b_0t}\prod\limits_{j=1}^N
(1-e^{-b_j t}) f(t)\,dt/t^M<\infty, \;s>0.
\end{equation}
Denote this non-negative distribution by $-\log\beta_{M,N}(b)$ so
that $\beta_{M,N}(b)\in(0, \,1]$ and
\begin{equation}
{\bf E}\Bigl[\exp\bigl(q\log\beta_{M,N}(b)\bigr)\Bigr] =
\eta_{M,N}(q\,|\,b), \;\Re(q)>-b_0.
\end{equation}
This is equivalent to \eqref{LKH}.
\end{proof}
\begin{proof}[Proof of Corollary \ref{Structure}]
We note that
\begin{equation}
\int\limits_0^\infty dK^{(f)}_{M,N}(t\,|\,b)/t <
\infty\;\text{iff}\; M<N.
\end{equation}
It follows from Proposition 4.13 in Chapter 3 of \cite{SteVHar} that
$\log\beta_{M,N}(b)$ is absolutely continuous if $M=N.$ If $M<N,$
$-\log\beta_{M,N}(b)$ is compound Poisson by Theorem \ref{main} and
Proposition 4.4 in Chapter 3 of \cite{SteVHar}. In particular,
\begin{equation}\label{lambda}
{\bf P}\bigl[\log\beta_{M,N}(b)=0\bigr] =
\exp\Bigl(-\int\limits_0^\infty dK^{(f)}_{M,N}(t\,|\,b)/t\Bigr).
\end{equation}
The result follows from \eqref{Kdef}.
\end{proof}
%%\begin{proof}[Proof of Corollary \ref{momentproblem}]
%%The moments satisfy Carleman's criterion by Theorem
%%ref{Asymptotics} (or,
%%It is sufficient to note that $\beta_{M,N}$ is compactly supported. %%confer Section 2.2 in \cite{Char}.
%%\end{proof}
\begin{proof}[Proof of Theorem \ref{Asymptotics}]
The starting point of the proof is \eqref{asym}. Substituting
\eqref{asym} into \eqref{eta} and using linearity of
$\mathcal{S}_N,$ we can write in the limit of $q\rightarrow \infty,$
$|\arg(q)|<\pi,$
\begin{align}
\eta_{M,N}(q\,|\,b) & = \exp\Bigl(-\bigl(\mathcal{S}_N
L_M\bigr)(0\,|\,b)\Bigr)\exp\Bigl(-\frac{1}{M!}
\mathcal{S}_N\bigl(B^{(f)}_M(w)\,\log(w)\bigr)(q\,|\,b)+\nonumber\\
& + \sum\limits_{k=0}^M \frac{B^{(f)}_k(0)
\bigl(\mathcal{S}_N(-w)^{M-k}\bigr)(q\,|\,b)}{k!(M-k)!}\sum\limits_{l=1}^{M-k}
\frac{1}{l}+O(q^{-1})\Bigr).\label{sum2}
\end{align}
Now, to compute
$\mathcal{S}_N\bigl(B^{(f)}_M(w)\,\log(w)\bigr)(q\,|\,b),$ we expand
the logarithm in powers of $1/q,$ resulting in terms of the form
\eqref{line2} with $n=M.$ By \eqref{line1} in Lemma \ref{mylemma},
if $r+m>M,$ then such terms are of order $O(1/q).$ If $r+m\leq M$
and $M<N,$ they are all zero by \eqref{line1}. If $r+m\leq M$ and
$M=N,$ the only non-zero terms satisfy $r+m=N$ so that they have
degree zero in $q.$ Hence, we have the estimate
\begin{align}
\mathcal{S}_N\bigl(B^{(f)}_M(w)\,\log(w)\bigr)(q\,|\,b) & =
\log(q)\,\mathcal{S}_N\bigl(B^{(f)}_M(w)\bigr)(q\,|\,b) + O(q^{-1}),
\;
{\rm if}\; M<N, \label{estimate}\\
& = \log(q)\,\mathcal{S}_N\bigl(B^{(f)}_M(w)\bigr)(q\,|\,b) + O(1),
\; {\rm if}\; M=N.
\end{align}
If $M<N,$ the expression in \eqref{estimate} is zero by
\eqref{mainID} and the sum in \eqref{sum2} is zero by \eqref{auxID}
so that \eqref{ourasym} follows from \eqref{sum2}. If $M=N,$ the
result follows from \eqref{mainIDN} and \eqref{auxIDN}.
\end{proof}
\begin{proof}[Proof of Corollary \ref{Inequality}]
The result follows from \eqref{ourasym} by letting $q\rightarrow
+\infty$ in \eqref{Keq} and recalling \eqref{lambda}.
\end{proof}
\begin{proof}[Proof of Theorem \ref{FunctEquat}]
It is sufficient to substitute \eqref{feq}, written in the form
\begin{equation}
L_M\bigl(w+a_i\,|\,a\bigr) = L_{M}(w\,|\,a) -
L_{M-1}(w\,|\,\hat{a}_i),
\end{equation}
into \eqref{eta} and recall the definition of
$\eta_{M-1,N}(q\,|\,\hat{a}_i, b).$
\end{proof}
\begin{proof}[Proof of Corollary \ref{FunctSymmetry}]
To prove \eqref{fe2}, note that Definition \ref{Soperator} implies
the identity
\begin{equation}
(\mathcal{S}_Nf)(q\,|\,b_0+x) = (\mathcal{S}_Nf)(q+x\,|\,b),
\end{equation}
and the result follows from Definition \ref{bdef}. \eqref{fe3} is
immediate from Definition \ref{bdef}. \eqref{fe1eq} is equivalent to
\eqref{fe1} due to the special case of $q=0$ in \eqref{fe1},
\begin{equation}\label{auxidentity}
\eta_{M, N}(a_i\,|\,a, b) = \exp\bigl(-(\mathcal{S}_N
L_{M-1})(0\,|\,\hat{a}_i, b)\bigr).
\end{equation}
The proof of \eqref{fe4} follows from \eqref{fe3}, \eqref{fe2}, and
\eqref{fe1eq}, in this order.
\begin{align}
\eta_{M, N}(q\,|\,a,\,b_j+a_i) & = \frac{\eta_{M, N-1}(q\,|\,a,\,\hat{b}_j)}{\eta_{M, N-1}(q\,|\,a,\,b_0+b_j+a_i, \hat{b}_j)}, \nonumber \\
& = \eta_{M, N-1}(q\,|\,a,\,\hat{b}_j)\frac{\eta_{M, N-1}(a_i\,|\,a,\,b_0+b_j, \hat{b}_j)}{\eta_{M, N-1}(q+a_i\,|\,a,\,b_0+b_j, \hat{b}_j)}, \nonumber \\
& =  \eta_{M-1, N-1}(q\,|\,\hat{a}_i,\,b_0+b_j,
\hat{b}_j)\frac{\eta_{M, N-1}(q\,|\,a,\,\hat{b}_j)}{\eta_{M,
N-1}(q\,|\,a,\,b_0+b_j, \hat{b}_j)}.
\end{align}
The result follows by yet another application of \eqref{fe2} and
\eqref{fe3}. Finally, to verify \eqref{funceqsymmetry}, we combine
\eqref{fe3} and \eqref{fe1eq} to obtain
\begin{align}
\eta_{M, N}(q+a_i\,|\,a,\,b) & = \eta_{M, N}(q\,|\,a,\,b) \,\eta_{M,
N}(a_i\,|\,a,\,b) \frac{\eta_{M-1, N-1}(q\,|\,\hat{a}_i,\,b_0+b_j,
\hat{b}_j)}{\eta_{M-1, N-1}(q\,|\,\hat{a}_i,\,\hat{b}_j)}, \nonumber
\\ & = \eta_{M, N}(q\,|\,a,\,b) \,\frac{\eta_{M,
N}(a_i\,|\,a,\,b)}{\eta_{M-1, N-1}(b_j\,|\,\hat{a}_i, \hat{b}_j)}
\frac{\eta_{M-1, N-1}(q+b_j\,|\,\hat{a}_i, \hat{b}_j)}{\eta_{M-1,
N-1}(q\,|\,\hat{a}_i,\,\hat{b}_j)}
\end{align}
by \eqref{fe2}. It remains to notice that \eqref{fe3} and
\eqref{auxidentity} imply
\begin{equation}
\eta_{M, N}(a_i\,|\,a,\,b) = \eta_{M-1, N-1}(b_j\,|\,\hat{a}_i,
\hat{b}_j).
\end{equation}
\end{proof}
\begin{proof}[Proof of Corollary \ref{moments}]
Repeated application of \eqref{fe1} gives the identity
\begin{equation}
\eta_{M,N}(q+k a_i\,|\,a, b) = \eta_{M,N}(q\,|\,a, b)
\exp\Bigl(-\sum_{l=0}^{k-1} \bigl(\mathcal{S}_N
L_{M-1}\bigr)(q+la_i\,|\,\hat{a}_i, b)\Bigr).
\end{equation}
Equations \eqref{posmom} and \eqref{negmom} now follow by letting
$q=0$ and $q=-ka_i,$ respectively.
\end{proof}
\begin{proof}[Proof of Corollary \ref{laplace}]
The absolute convergence of the series in \eqref{powerseries}
follows from Theorem \ref{Asymptotics} and \eqref{posmom}. Its
equality to the Laplace transform is the general property of the
power series of positive integral moments, confer Section 7.6 of
\cite{Feller}.
\end{proof}
\begin{proof}[Proof of Corollary \ref{Rama}]
The proof is a direct corollary of Ramanujan's Master Theorem,
confer \cite{Ramanujan}. It is only sufficient to note that
$\eta_{M,N}(q\,|\,a, b)$ is analytic over $\Re(q)>-b_0$ and, by
Theorem \ref{Asymptotics}, satisfies Hardy's growth conditions
there.
\end{proof}
\begin{proof}[Proof of Theorem \ref{Factorization}]
It is sufficient to verify \eqref{infinprod1} as \eqref{infinprod2}
and \eqref{infinprod3} are equivalent by \eqref{fe2} and
\eqref{fe3}. Let $\Re(q)>-b_0.$ Consider the product on the
right-hand side of \eqref{infinprod1} and reduce it by means of the
L\'evy-Khinchine representation for $\eta_{M-1, N}(q\,|\,\hat{a}_i,
b).$
\begin{equation}
\prod\limits_{k=0}^L \eta_{M-1,N}(q\,|\,\hat{a}_i, b_0+ka_i) =
\exp\Bigl(\int\limits_0^\infty \frac{dt}{t} (e^{-qt}-1)\frac{e^{-b_0
t}\prod_{j=1}^N (1-e^{-b_j t})}{\prod_{j\neq i}^M (1-e^{-a_j t})}
\frac{1-e^{-a_i t(L+1)}}{1-e^{-a_i t}}\Bigr).
\end{equation}
Letting $L\rightarrow \infty,$ we obtain %%the L\'evy-Khinchine representation of $\eta_{M,N}(q\,|\,a)$
\eqref{LKH} by dominated convergence. %%This completes the proof for $\Re(q)>-b_0.$
%%As the convergence of the infinite product in \eqref{infinprod} is uniform on compact
%%subsets, the product is holomorphic over $\Re(q)>-b_0.$
As $\eta_{M,N}(q\,|\,a, b)$ is holomorphic over
$q\in\mathbb{C}-(-\infty, -b_0],$ \eqref{infinprod1} holds there
%%for $q\in\mathbb{C}-(-\infty, -b_0]$
by analytic continuation.
\end{proof}
\begin{proof}[Proof of Corollary \ref{solution}]
By the infinite product representation in \eqref{infinprod2}, we
have
%% for $\eta_{M,N}(q+a_i\,|\,a)$ is
%%We need to show that the infinite
%%product representation of $\eta_{M,N}(q\,|\,a)$ in \eqref{infinprod}
%%satisfies the functional equation in Theorem \ref{FunctEquat}.
\begin{align}
\eta_{M,N}(q+a_i\,|\,a) & =
\lim\limits_{K\rightarrow\infty}\Bigl[\prod\limits_{k=0}^K
\frac{\eta_{M-1,N}(q+(k+1) a_i\,|\,\hat{a}_i, b)}{\eta_{M-1,N}(k
a_i\,|\,\hat{a}_i, b)}\Bigr], \nonumber \\
& = \lim\limits_{K\rightarrow\infty}
\Bigl[\frac{\eta_{M-1,N}\bigl((K+1) a_i\,|\,\hat{a}_i,
b\bigr)}{\eta_{M-1,N}(q\,|\,\hat{a}_i, b)} \prod\limits_{k=0}^{K+1}
\frac{\eta_{M-1,N}(q+k a_i\,|\,\hat{a}_i, b)}{\eta_{M-1,N}(k
a_i\,|\,\hat{a}_i, b)}\Bigr], \nonumber \\
& = \eta_{M, N}(q\,|\,a, b) \exp\bigl(-(\mathcal{S}_N
L_{M-1})(q\,|\,\hat{a}_i, b)\bigr)
\end{align}
by \eqref{ourasym}. %%This proves \eqref{fe1}. The argument for
%%\eqref{fe4} goes through verbatim.
Incidently, the same argument shows that \eqref{infinprod3} is the
solution to \eqref{fe4}.
\end{proof}
\begin{proof}[Proof of Theorem \ref{Reduction}]
It is sufficient to show that we have the identity
\begin{equation}\label{auxredeq}
\eta_{M, N}(q\,|\,a, b_j=n\,a_i) = \prod_{k=0}^{n-1} \eta_{M-1,
N-1}\bigl(q\,|\,\hat{a}_i, b_0+ka_i,\, \hat{b}_j\bigr).
\end{equation}
This is done by induction on $n.$ If $n=1,$ then the result follows
from \eqref{LKH}. Assume \eqref{auxredeq} holds for $n-1,$
\emph{i.e.} $b_j=(n-1)\,a_i.$ By \eqref{fe4}, we have
\begin{align}
\eta_{M, N}(q\,|\,a,\,b_j=(n-1)\,a_i+a_i) & = \eta_{M,
N}(q\,|\,a,\,b_j=(n-1)\,a_i)\,
%%\eta_{M-1, N-1}(q\,|\,\hat{a}_i,\,b_0+b_j, \hat{b}_j). \label{fe4}
\frac{\eta_{M-1, N-1}(q+b_j\,|\,\hat{a}_i,\,\hat{b}_j)} {\eta_{M-1,
N-1}(b_j\,|\,\hat{a},\,\hat{b}_j)}, \nonumber \\
& = \prod_{k=0}^{n-1} \eta_{M-1, N-1}\bigl(\hat{a}_i, b_0+ka_i,\,
\hat{b}_j\bigr),
\end{align}
by the induction assumption and \eqref{fe2}.
\end{proof}
\begin{proof}[Proof of Theorem \ref{Momentsaspecial}]
In the case of $a_i=1$ we can write \eqref{posmom} in the form
\begin{equation}\label{momentfunceq}
\bigl(\mathcal{S}_N L_{M}\bigr)(k\,|\, b) = \bigl(\mathcal{S}_N
L_{M}\bigr)(0\,|\, b)-\sum\limits_{l=0}^{k-1} \bigl(\mathcal{S}_N
L_{M-1}\bigr)(l\,|\,b).
\end{equation}
By repeated application of \eqref{momentfunceq} and the identity
\begin{equation}
\sum\limits_{i_1=0}^{n-1} \sum\limits_{i_2=0}^{i_1-1}\cdots
\sum\limits_{i_k=0}^{i_{k-1}-1} 1 = \binom{n}{k},
\end{equation}
we obtain by induction on $k=0\cdots M-1,$
\begin{align}
\bigl(\mathcal{S}_N L_{M}\bigr)(n\,|\, b) - \bigl(\mathcal{S}_N
L_{M}\bigr)(0\,|\, b) & = \sum\limits_{i=1}^k (-1)^i \binom{n}{i}
\bigl(\mathcal{S}_N L_{M-i}\bigr)(0\,|\, b) + \nonumber \\ & +
(-1)^{k+1} \sum\limits_{i_1=0}^{n-1} \sum\limits_{i_2=0}^{i_1-1}
\cdots \sum\limits_{i_{k+1}=0}^{i_{k}-1} \bigl(\mathcal{S}_N
L_{M-k-1}\bigr)(i_{k+1}\,|\, b).
\end{align}
The result follows by letting $k=M-1$ and recalling that
$L_0(w)=-\log(w).$
\end{proof}

\section{Conclusions}
We constructed and studied the main properties of a novel class of
what we called Barnes beta probability distributions $\beta_{M,N}(a,
b).$ $\beta_{M,N}(a, b)$ is a distribution on $(0, 1]$ that is
parameterized by two sets of positive real numbers
$a=(a_1,\cdots,a_M)$ and $b=(b_0,\cdots,b_N)$ and defined by its
Mellin transform $\eta_{M,N}(q\,|\,a, b).$ We gave four different
representations of
$\eta_{M,N}(q\,|\,a, b).$ %%of the Mellin transform of $\beta_{M,N}.$
The defining representation is in the form of a product of ratios of
Barnes multiple gamma functions $\Gamma_M(w\,|\,a),$ thereby
generalizing the classic beta distribution. We used Malmst\'en-type
formula of Ruijsenaars for $\log\Gamma_M(w\,|\,a)$ to show that
$-\log\beta_{M,N}(a, b)$ is infinitely divisible on $[0,\infty)$ by
deriving its L\'evy-Khinchine form and thus giving the $2$nd
representation of $\eta_{M,N}(q\,|\,a, b).$ %%the Mellin transform of $\beta_{M,N}.$
The L\'evy-Khinchine form allowed us to show that
$-\log\beta_{M,N}(a, b)$ is compound Poisson if $M<N$ and absolutely
continuous if $M=N.$
%%determine that $\beta_{M,N}$ is absolutely continuous iff
We used the functional equation of $\Gamma_M(w\,|\,a)$ to derive a
functional equation for $\eta_{M,N}(q\,|\,a, b)$ and thus to compute
the integral moments of $\beta_{M,N}(a, b)$ in the form of
Selberg-type products of $\Gamma_{M-1}(w\,|\,a).$ The Ruijsenaars
form of the asymptotic expansion of $\log\Gamma_M(w\,|\,a)$ in the
limit $w\rightarrow\infty$ gave us the asymptotic of
$\eta_{M,N}(q\,|\,a, b)$ in the limit $q\rightarrow\infty.$ We used
this asymptotic in the case of $M<N$ to give a probabilistic proof
of an inequality involving multiple gamma functions. We also used it
to show the convergence of the power series of moments of
$\beta_{M,N}(a, b)$ to the Laplace transform. The resulting series
of Selberg-type products of $\Gamma_{M-1}(w\,|\,a)$ is therefore
positive, giving an interesting application of the general theory.
We related this series to $\eta_{M,N}(q\,|\,a, b)$ by Ramanujan's
Master Theorem, giving the $3$rd representation of
$\eta_{M,N}(q\,|\,a, b).$ We solved the functional equation of
$\eta_{M,N}(q\,|\,a, b)$ in the form of Shintani-type infinite
products of $\eta_{M-1,N}(q\,|\,a, b)$ and
$\eta_{M-1,N-1}(q\,|\,a, b),$ resulting in the $4$th representation. %%$\beta_{M,N},$
Finally, we established several symmetries of the Mellin transform
in the form of functional equations relating $\eta_{M,N}(q\,|\,a,
b)$ to $\eta_{M,N-1}(q\,|\,a, b),$ $\eta_{M-1,N}(q\,|\,a, b),$ and
$\eta_{M-1,N-1}(q\,|\,a, b).$
%%and translated them into integral equations %% of the Mellin convolution type
%%for the probability density of $\beta_{M,M}(a, b).$

%%We considered two special cases of $a$ and $b.$

We illustrated our theory of Barnes beta distributions with several
examples. First, we considered two special cases of $a$ and $b.$
%%for general $M$ and $N.$
If $b_N$ is an integer multiple of $a_M,$ $\beta_{M,N}(a, b)$
decomposes into a product of $\beta_{M-1,N-1}(a, b)s.$ If $a_i=1$
for all $i,$ the moments of $\beta_{M,N}(a, b)$ are given by a
multiple product generalizing the moments of the classic beta
distribution. Second, in some elementary cases of small $M$ and $N$
we computed the density and weight at 1 of $\beta_{M,N}(a, b)$
exactly. Our main
non-elementary example is $\beta_{2,2}(a, b),$ %%which is the simplest non-trivial case of
%%a Barnes beta distribution,
in which case our formulas simplify to expressions involving Euler's
gamma function.

The main area of applications that we considered in this paper is
the probabilistic theory of the Selberg integral. We constructed a
distribution with the property that its moments are given by
Selberg's formula and decomposed it into a product of
$\beta_{2,2}^{-1}(a, b)s.$ This construction leads to a
probabilistic interpretation of the Selberg integral.

%%%%%%%%%%%%%%%%%%%%%%%%%%%%%%%%%%%%%%%%%%%%%%%%%%%%%%%%%%%%%%%%%%%
%%                                                               %%
%% You may add acknowledgments (optional).                       %%
%%                                                               %%
%%%%%%%%%%%%%%%%%%%%%%%%%%%%%%%%%%%%%%%%%%%%%%%%%%%%%%%%%%%%%%%%%%%

%%\ACKNO{We thank Martin Hairer who provided a nice \texttt{MR}
%%macro.}

%%%%%%%%%%%%%%%%%%%%%%%%%%%%%%%%%%%%%%%%%%%%%%%%%%%%%%%%%%%%%%%%%%%
%%                                                               %%
%% You have reached the end of your document.                    %%
%%                                                               %%
%%%%%%%%%%%%%%%%%%%%%%%%%%%%%%%%%%%%%%%%%%%%%%%%%%%%%%%%%%%%%%%%%%%


\begin{thebibliography}{99}

\bibitem{Ramanujan} Amdeberhan, T., Espinosa, O., Gonzalez, I., Harrison, M., Moll, V.,
Straub, A.: Ramanujan's Master Theorem. \emph{Ramanujan J.}
\textbf{29}, (2012), 103--120.

\bibitem{mBarnes} Barnes, E. W.: On the theory of the multiple gamma
function. \emph{Trans. Camb. Philos. Soc.} \textbf{19}, (1904),
374--425.

\bibitem{BiaPitYor} Biane, P., Pitman, J., Yor, M.: Probability laws
related to the Jacobi theta and Riemann zeta functions, and brownian
excursions. \emph{Bulletin of the American Mathematical Society}
\textbf{38}, (2001), 435--465.

\bibitem{ChaLet} Chamayou, J-F. and Letac, G.: Additive properties of the Dufresne laws and their
multivariate extension. \emph{J. Theoret. Probab.} \textbf{12},
(1999), 1045--1066.

\bibitem{Duf10} Dufresne, D.: G distributions and the beta-gamma
algebra. \emph{Elect. J. Probab.} \textbf{15}, (2010), 2163--2199.

\bibitem{Feller} Feller, W.  An Introduction to Probability Theory and Its
Applications, vol. II, 2nd ed. \emph{John Wiley and Sons}, New York,
1971.

\bibitem{ForresterBook} Forrester, P. J.:  Log-Gases and Random
Matrices. \emph{Princeton University Press}, Princeton, 2010.

\bibitem{FLDR} Fyodorov, Y. V., Le Doussal, P., Rosso, A.: Statistical mechanics of logarithmic REM:
duality, freezing and extreme value statistics of 1/f noises
generated by gaussian free fields. \emph{J. Stat. Mech. Theory Exp},
(2009), no. 10, P10005.

\bibitem{HubKuz} Hubalek, F. and Kuznetsov, A.: A convergent series representation for the density of the supremum
of a stable process, \emph{Elect. Comm. in Probab.} \textbf{16},
(2011), 84--95.

\bibitem{Jacod} Jacod, J., Kowalski, E., Nikeghbali, A.:
Mod-gaussian convergence: new limit theorems in probability and
number theory. \emph{Forum Mathematicum} \textbf{23}, (2011),
835--873.

\bibitem{KataOhts} Katayama, K. and Ohtsuki, M. On the multiple
gamma-functions. \emph{Tokyo J. Math.} \textbf{21}, (1998),
159--182.

\bibitem{Kuz} Kuznetsov, A. On extrema of stable processes. \emph{Ann. Probab.} \textbf{39}, (2011), 1027--1060.

\bibitem{KuzPar} Kuznetsov, A. and Pardo, J. C. Fluctuations of
stable processes and exponential functionals of hypergeometric
L\'evy processes. \emph{Acta Applicandae Mathematicae} May (2012).

\bibitem{LagRains} Lagarias, J. and Rains, E. On a two-variable zeta function for number
fields. \emph{Ann. Inst. Fourier, Grenoble} \textbf{53}, (2003),
1--68.

\bibitem{NikYor} Nikeghbali, A. and Yor, M.: The Barnes G function
and its relations with sums and products of generalized gamma
convolutions variables. \emph{Elect. Comm. in Prob.} \textbf{14},
(2009), 396--411.

\bibitem{Me4} Ostrovsky, D.: Mellin transform of the limit
lognormal distribution. \emph{Comm. Math. Phys.} \textbf{288},
(2009), 287--310.

\bibitem{Me} Ostrovsky, D. Selberg integral as a meromorphic
function. \emph{Int. Math. Res. Notices} \textbf{17}, (2013), 
3988--4028.

\bibitem{Ruij} Ruijsenaars, S. N. M.: On Barnes' multiple zeta and
gamma functions. \emph{Advances in Mathematics} \textbf{156},
(2000), 107--132.

\bibitem{Selberg} Selberg, A.: Remarks on a multiple integral. \emph{Norske Mat.
Tidsskr.} \textbf{26}, (1944), 71--78.

\bibitem{Shintani} Shintani, T.: A proof of the classical
Kronecker limit formula. \emph{Tokyo J. Math.} \textbf{3}, (1980),
191--199.

\bibitem{SteVHar} Steutel, F. W.  and van Harn, K.:  Infinite Divisibility of Probability Distributions on the Real
Line. \emph{Marcel Dekker}, New York, 2004.

\bibitem{Yor} Yor, M.: A Note about Selberg's integrals in relation
to the beta-gamma algebra, in Advances in Mathematical Finance, Fu,
M., Jarrow, R., Yen, J., editors, \emph{Birkhauser}, Boston, 2007,
49--58.



\end{thebibliography}
\end{document}